\documentclass[reqno]{amsart}
\usepackage{amsfonts}
\usepackage{amssymb}
\usepackage{amsmath}
\usepackage{amsthm}
\usepackage{amscd}
\usepackage{mathrsfs}
\usepackage{graphicx, color}
\usepackage{url}
\usepackage[stable]{footmisc}

\theoremstyle{plain}
\newtheorem{theorem}{Theorem}[section]
\newtheorem{lemma}[theorem]{Lemma}

\newtheorem{proposition}[theorem]{Proposition}

\theoremstyle{definition}

\newtheorem{remark}[theorem]{Remark}

\allowdisplaybreaks[1]

\newcommand{\RotE}{\mathrm{Rot} \hspace{0.1em} \mathbb{E}^{2}}

\begin{document}

\title[R-equivalence and inner automorphisms]{R-equivalence of quandle colorings and \\ inner automorphisms}
\author{Mai Sato}
\address{Department of Mathematics, Tsuda University, 2-1-1 Tsuda-machi, Kodaira-shi, Tokyo 187-8577, Japan}
\email{md2461sm@gm.tsuda.ac.jp}


\begin{abstract}
R-equivalence is an equivalence relation on the set of colorings of an oriented knot diagram by a quandle.
In this paper, we show that a certain subgroup of the inner automorphism group of a quandle acts on the R-equivalence class of a given coloring by the quandle.
We also determine the R-equivalence classes of colorings of a diagram of a $2$-bridge knot by a dihedral quandle completely, under a certain condition.
\end{abstract}

\maketitle

\section{Introduction}
\label{sec:introduction}

A quandle, introduced by Joyce \cite{J1982}, is an algebraic system whose axioms have close relationships with Reidemeister moves for oriented knot diagrams.
Although it is called a distributive groupoid instead of a quandle, the same notion was also introduced by Matveev \cite{M1984}.
For each quandle $X$, we may consider $X$-colorings of an oriented knot diagram.
A finite sequence of Reidemeister moves and planar isotopies naturally induces a bijection between the $X$-colorings of the original diagram and those of the resulting diagram.
The number of $X$-colorings gives us an elementary but useful knot invariant.

We may deform a diagram to itself by a finite sequence of Reidemeister moves and planar isotopies.
On the other hand, the bijection induced from the sequence may send an $X$-coloring of the diagram to a different one.
To study this phenomenon, the author defined those colorings to be R-equivalent \cite{S2026}.\footnote{We note that an essentially equivalent notion was introduced by Lopes \cite{L2003} in a different context: he showed that, in our terminology, the set of R-equivalence classes of $X$-colorings of a diagram is an invariant of the knot. This invariant refines the number of $X$-colorings. From his view point, this paper deals with the computation of this invariant.}
In \cite{S2026}, she focused on the quandle $\RotE$, and completely determined the R-equivalence classes of $\RotE$-colorings of a diagram of a torus knot under a certain condition.
She introduced several key deformations of diagrams of torus knots, and showed that any pair of R-equivalent $\RotE$-colorings is related by a finite sequence of these deformations.

Finding key deformations of diagrams, which might relate any two R-equivalent colorings, is an effective way to study R-equivalence.
On the other hand, it depends heavily on the given diagram and quandle.
Thus, in this paper, we first study a sufficient condition for two colorings to be R-equivalent, which holds for an arbitrary diagram and quandle.
More precisely, we show that a certain subgroup of the inner automorphism group of $X$ acts on the R-equivalence class of a given $X$-coloring (Theorem \ref{R-equivalence_and_inner_automorphism}).
This condition is not necessary in general (see Section \ref{sec:R-equivalence_classes_of_R_p-colorings_of_2-bridge_knot}).
However, for a suitable $2$-bridge knot and dihedral quandle, it is also a necessary condition (Theorem \ref{necessary_and_sufficient_condition}).
We show this by giving the complete classification of the R-equivalence classes of colorings of a diagram of a $2$-bridge knot by a dihedral quandle (Theorem \ref{R-equivalence_classes_of_R_p-colorings_of_2-bridge_knot}).

This paper is organized as follows.
In Section \ref{sec:preliminaries}, we review $X$-colorings and their R-equivalence.
We prove Theorem \ref{R-equivalence_and_inner_automorphism} in Section \ref{sec:R-equivalence_and_inner_automorphism}.
Finally, in Section \ref{sec:R-equivalence_classes_of_R_p-colorings_of_2-bridge_knot}, we show Theorems \ref{necessary_and_sufficient_condition} and \ref{R-equivalence_classes_of_R_p-colorings_of_2-bridge_knot}.

\section{Preliminaries}
\label{sec:preliminaries}

In this section, we review the notion of R-equivalence.
To do it, we start with recalling the definition of a quandle.
For details on quandles, we refer the reader to \cite{K2017}.

A \emph{quandle} is a non-empty set $X$ equipped with a binary operation $\ast : X \times X \to X$ satisfying the following three axioms.
\begin{itemize}
 \item[Q1.] For any $x \in X$, $x \ast x = x$.
 \item[Q2.] For any $y \in X$, the map $S_{y} : X \to X$ given by $S_{y}(x) = x \ast y$ is bijective.
 \item[Q3.] For any $x, y, z \in X$, $\left( x \ast y \right) \ast z = \left( x \ast z \right) \ast \left( y \ast z \right)$.
\end{itemize}
For each $x, y \in X$ and $\varepsilon \in \{ \pm 1 \}$, we let $x \ast^{\varepsilon} y$ denote $S^{\varepsilon}_{y}(x)$.

The notion of a homomorphism is appropriately defined for quandles.
Let $X$ be a quandle, and $\mathrm{Aut}(X)$ denote the \emph{automorphism group} of $X$.
By the axioms Q2 and Q3, for each $y \in X$, $S_{y}$ is an automorphism of $X$.
Define $\mathrm{Inn}(X)$ to be the subgroup of $\mathrm{Aut}(X)$ generated by the automorphisms $S_{y}$ ($y \in X$).
We call an element of $\mathrm{Inn}(X)$ an \emph{inner automorphism} of $X$, and $\mathrm{Inn}(X)$ the \emph{inner automorphism group} of $X$.
The inner automorphism group $\mathrm{Inn}(X)$ acts on $X$ from the right by $x \cdot S_{y} = x \ast y$ ($x, y \in X$).
We say that $X$ is \emph{connected} if $\mathrm{Inn}(X)$ acts on $X$ transitively.

A subset $Y$ of $X$ is said to be a \emph{subquandle} of $X$ if $Y$ is also a quandle with respect to the binary operation on $X$ restricted to $Y$.
We note that any subset of $X$ consisting of a single element becomes a subquandle of $X$.
For each subset $Z$ of $X$, there exists a minimal subquandle of $X$ which includes $Z$.
We call this the subquandle of $X$ \emph{generated by $Z$}, and write it as $\langle Z \rangle$.

Let $D$ be an oriented knot diagram.
A map $\mathscr{C} : \{ \text{all arcs of }D \} \to X$ is called an \emph{$X$-coloring} of $D$ if $\mathscr{C}$ satisfies the condition depicted in Figure \ref{coloring_condition} at each crossing of $D$.
In the figure, $x$, $y$ and $x \ast y$ denote the elements of $X$ assigned to the corresponding arcs by $\mathscr{C}$.
We call them the \emph{colors} of the arcs.
We note that any constant map $\mathscr{C}$ satisfies the above condition.
We thus call such a map a \emph{trivial $X$-coloring} of $D$.
We let $\mathrm{Col}_{X}(D)$ denote the set of $X$-colorings of $D$.
\begin{figure}[htbp]
 \centering
 \includegraphics[scale=0.5]{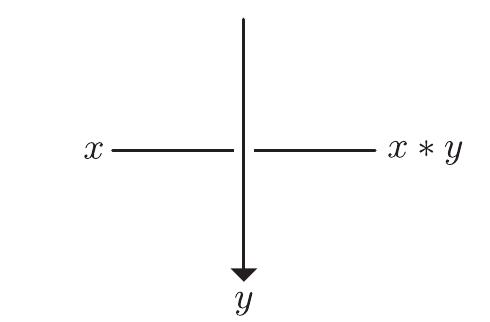}
 \caption{The condition for an $X$-coloring.}
 \label{coloring_condition}
\end{figure}

Suppose that a Reidemeister move RI, RII or RIII takes $D$ to a diagram $D^{\prime}$.
Then, for each $X$-coloring $\mathscr{C}$ of $D$, we have a unique $X$-coloring $\mathscr{C}^{\prime}$ of $D^{\prime}$ which assigns the same colors with $\mathscr{C}$ for the arcs unrelated to the deformation and consistent colors for the others as depicted in Figure \ref{colored_Reidemeister_moves}.
Therefore, we obtain a bijection $\mathrm{Col}_{X}(D) \to \mathrm{Col}_{X}(D^{\prime})$.
We note that the axioms of a quandle guarantee the existence and uniqueness of $\mathscr{C}^{\prime}$.
Suppose instead that a planar isotopy takes $D$ to $D^{\prime}$.
Then, by assigning the same colors to the corresponding arcs, we obtain a bijection $\mathrm{Col}_{X}(D) \to \mathrm{Col}_{X}(D^{\prime})$.
In conclusion, a finite sequence of Reidemeister moves and planar isotopies taking $D$ to $D^{\prime}$ induces a bijection $\mathrm{Col}_{X}(D) \to \mathrm{Col}_{X}(D^{\prime})$.
\begin{figure}[htbp]
 \centering
 \includegraphics[scale=0.5]{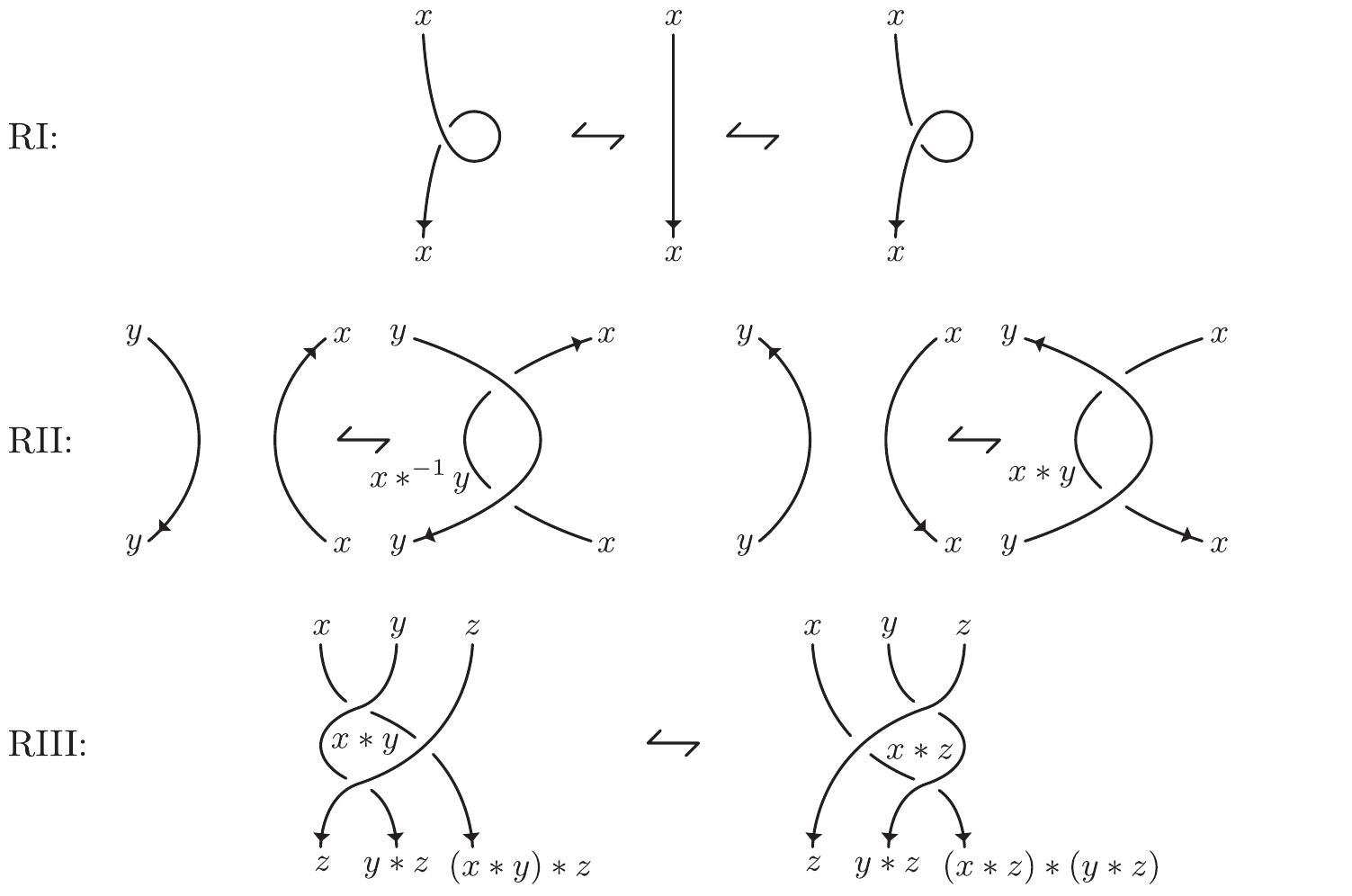}
 \caption{Reidemeister moves induce bijections $\mathrm{Col}_{X}(D) \to \mathrm{Col}_{X}(D^{\prime})$.}
 \label{colored_Reidemeister_moves}
\end{figure}

Let $\mathscr{C}_{1}$ and $\mathscr{C}_{2}$ be $X$-colorings of a diagram $D$.
The coloring $\mathscr{C}_{1}$ is said to be \emph{R-equivalent} to $\mathscr{C}_{2}$ if there exists a finite sequence of Reidemeister moves and planar isotopies sending $\mathscr{C}_{1}$ to $\mathscr{C}_{2}$.
Clearly, R-equivalence yields an equivalence relation on $\mathrm{Col}_{X}(D)$.
We note that a trivial $X$-coloring of $D$ is not R-equivalent to any other $X$-colorings of $D$.
We also note that the concept of R-equivalence does not depend on the choice of a diagram of a given oriented knot.
More precisely, if a finite sequence of Reidemeister moves and planar isotopies taking $D$ to $D^{\prime}$ sends $X$-colorings $\mathscr{C}_{1}$ and $\mathscr{C}_{2}$ of $D$ to $X$-colorings $\mathscr{C}^{\prime}_{1}$ and $\mathscr{C}^{\prime}_{2}$ of $D^{\prime}$, respectively, then $\mathscr{C}^{\prime}_{1}$ and $\mathscr{C}^{\prime}_{2}$ are R-equivalent if and only if $\mathscr{C}_{1}$ and $\mathscr{C}_{2}$ are.

\section{A sufficient condition for R-equivalence}
\label{sec:R-equivalence_and_inner_automorphism}

For a quandle $X$ and an oriented knot diagram $D$, we may define a right action of $\mathrm{Inn}(X)$ on $\mathrm{Col}_{X}(D)$ by $\mathscr{C} \cdot S = S \circ \mathscr{C}$.
The aim of this section is to establish the following theorem.

\begin{theorem}
 \label{R-equivalence_and_inner_automorphism}
 Let $X$ be a quandle, $D$ an oriented knot diagram, and $\mathscr{C}$ an $X$-coloring of $D$.
 For any $S \in \mathrm{Inn}( \langle \mathrm{Im} \, \mathscr{C} \rangle )$, $\mathscr{C}$ is R-equivalent to $\mathscr{C} \cdot S$.
\end{theorem}

To prove this theorem, we prepare the following lemma.

\begin{lemma}
 \label{lemma_for_R-equivalence_and_inner_automorphism}
 Let $X$ be a quandle, $D$ an oriented knot diagram, and $\mathscr{C}$ an $X$-coloring of $D$.
 For any $y \in \langle \mathrm{Im} \, \mathscr{C} \rangle$ and $\varepsilon \in \{ \pm 1 \}$, $\mathscr{C}$ is R-equivalent to $\mathscr{C} \cdot S^{\varepsilon}_{y}$.
\end{lemma}

\begin{proof}
 Suppose that $\mathrm{Im} \, \mathscr{C} = \{ x_{1}, x_{2}, \cdots, x_{n} \}$.
 Then $y$ can be written as $y = ( \cdots (( x_{i_{0}} \ast^{\varepsilon_{1}} x_{i_{1}} ) \ast^{\varepsilon_{2}} x_{i_{2}}) \cdots ) \ast^{\varepsilon_{m}} x_{i_{m}}$, where  each $i_{k}$ is an integer with $1 \leq i_{k} \leq n$ and $\varepsilon_{k} \in \{ \pm 1 \}$ (see \cite[Lemma 8.5.5]{K2017}).
 
 For each integer $k$ with $0 \leq k \leq m$, let $y_{k} = ( \cdots (( x_{i_{0}} \ast^{\varepsilon_{1}} x_{i_{1}} ) \ast^{\varepsilon_{2}} x_{i_{2}}) \cdots ) \ast^{\varepsilon_{k}} x_{i_{k}}$.
 Suppose that $D_{0} = D$ and $\mathscr{C}_{0} = \mathscr{C}$.
 We note that $D_{0}$ with $\mathscr{C}_{0}$ has arcs colored $y_{0} = x_{i_{0}}$ and $x_{i_{1}}$.
 For each $k$ with $1 \leq k \leq m$, we construct a diagram $D_{k}$ with an $X$-coloring $\mathscr{C}_{k}$ inductively as follows.
 Consider the diagram $D_{k-1}$ with $\mathscr{C}_{k-1}$.
 By  a finite sequence of RII, pull a part of an arc colored $y_{k-1}$ over other arcs of $D_{k-1}$, until the diagram locally looks like the one shown in the left-hand side of Figure \ref{deformations_of_lemma_3-2} (a) or (b).
  By applying one of the deformations depicted in Figure \ref{deformations_of_lemma_3-2}, we obtain $D_{k}$ with $\mathscr{C}_{k}$ having an arc colored $y_{k} = y_{k-1} \ast^{\varepsilon_{k}} x_{i_{k}}$.
 We note by construction that $D_{k}$ with $\mathscr{C}_{k}$ has arcs colored $x_{1}, x_{2}, \cdots, x_{n}$.
 
 Consider the diagram $D_{m}$ with $\mathscr{C}_{m}$.
 By a finite sequence of Reidemeister moves, sweep a part of an arc colored $y = y_{m}$ once around $S^{2}$, passing it over other arcs of $D_{m}$, until it returns to its original position.
 This operation sends $\mathscr{C}_{m}$ to $\mathscr{C}_{m} \cdot S^{\pm 1}_{y}$, where the sign depends on the direction of the sweep.
 Therefore, we obtain $D_{m}$ with $\mathscr{C}_{m} \cdot S^{\varepsilon}_{y}$.
 \footnote{We note that, in the case $\varepsilon = +1$, an essentially equivalent operation was also treated by Asami and Satoh \cite{AS2005} in a different context: they introduced it to compute weights of $X$-colorings of a diagram of a surface knot that arises from the twist-spinning of a (classical) knot.}
 By applying the above sequence of RII taking $D_{0}$ to $D_{m}$ in reverse order, we obtain $D = D_{0}$ with $\mathscr{C} \cdot S^{\varepsilon}_{y}$.
\begin{figure}[htbp]
 \centering
 \includegraphics[scale=0.5]{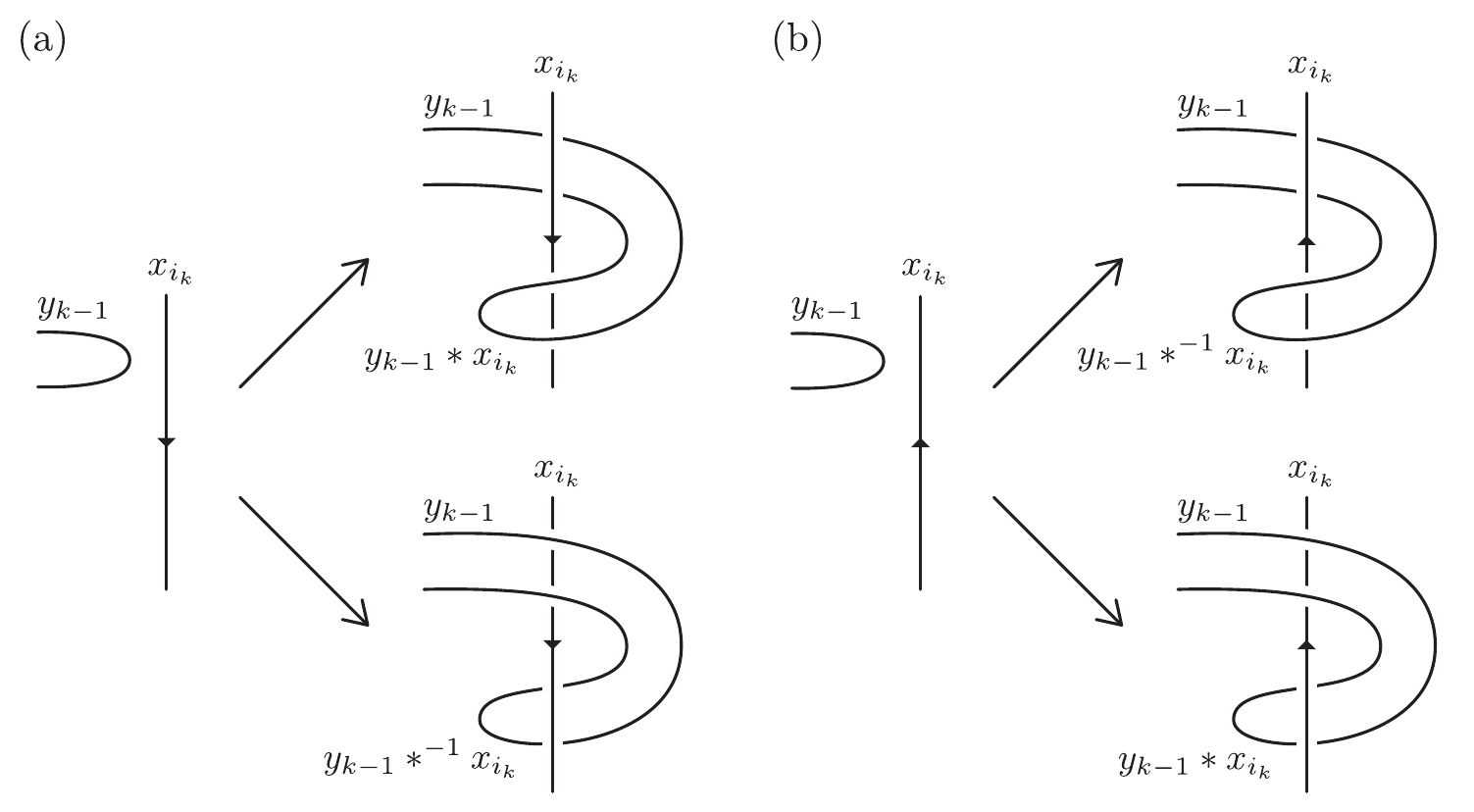}
 \caption{Deformations yielding arcs colored $y_{k} = y_{k-1} \ast^{\varepsilon_{k}} x_{i_{k}}$.}
 \label{deformations_of_lemma_3-2}
\end{figure}
\end{proof}

\begin{proof}[Proof of Theorem \ref{R-equivalence_and_inner_automorphism}]
 We may assume that $S = S^{\varepsilon_{1}}_{y_{1}} \cdots S^{\varepsilon_{m}}_{y_{m}}$ with $y_{k} \in \langle \mathrm{Im} \, \mathscr{C} \rangle$ and $\varepsilon_{k} \in \{ \pm 1 \}$.
 For each integer $k$ with $1 \leq k \leq m$, let $S_{k} = S^{\varepsilon_{1}}_{y_{1}} \cdots S^{\varepsilon_{k}}_{y_{k}}$.
 By Lemma \ref{lemma_for_R-equivalence_and_inner_automorphism}, $\mathscr{C}$ is R-equivalent to $\mathscr{C} \cdot S_{1}$.
 Assume that $\mathscr{C}$ is R-equivalent to $\mathscr{C} \cdot S_{k}$ ($1 \leq k \leq m - 1$).
 Since $S_{k}$ is an automorphism of $\langle \mathrm{Im} \, \mathscr{C} \rangle$, $\langle \mathrm{Im} \, \mathscr{C} \rangle$ is also generated by $S_{k}(\mathrm{Im} \, \mathscr{C}) = \mathrm{Im} \, \mathscr{C} \cdot S_{k}$.
 It follows that $\langle \mathrm{Im} \, \mathscr{C} \rangle = \langle \mathrm{Im} \, \mathscr{C} \cdot S_{k} \rangle$, and hence $y_{k+1} \in \langle \mathrm{Im} \, \mathscr{C} \cdot S_{k} \rangle$.
 Since $\mathscr{C} \cdot S_{k+1} = (\mathscr{C} \cdot S_{k}) \cdot S^{\varepsilon_{k+1}}_{y_{k+1}}$, $\mathscr{C} \cdot S_{k}$ is R-equivalent to $\mathscr{C} \cdot S_{k+1}$ by Lemma \ref{lemma_for_R-equivalence_and_inner_automorphism}.
 Thus $\mathscr{C}$ is R-equivalent to $\mathscr{C} \cdot S_{k+1}$.
\end{proof}

\begin{remark}
 \label{g-colored_knot}
 For a group $G$, a $G$-coloured knot is defined to be a knot together with an epimorphism from its knot group onto $G$.
 To study $G$-coloured knots up to ambient isotopy and ``twist moves'', Moskovich proved a result similar to Theorem \ref{R-equivalence_and_inner_automorphism} \cite[Lemma 2.2]{M2011}, utilizing an operation similar to the one considered in the last paragraph of the proof of Lemma \ref{lemma_for_R-equivalence_and_inner_automorphism}.
 We note that $G$ equipped with a binary operation $\ast$ given by $g \ast h = h^{-1}gh$ becomes a quandle.
 In the special cases where $G$ is normally generated by an element and $\mathscr{C}$ is a $(G, \ast)$-coloring of $D$ satisfying $\langle \mathrm{Im} \, \mathscr{C} \rangle = G$ as groups, Theorem \ref{R-equivalence_and_inner_automorphism} can also be obtained from \cite[Lemma 2.2]{M2011}.
\end{remark}

\section{R-equivalence classes of colorings of a $2$-bridge knot by \\ a dihedral quandle}
\label{sec:R-equivalence_classes_of_R_p-colorings_of_2-bridge_knot}

For each positive integer $q \geq 3$, the cyclic group $\mathbb{Z} / q \mathbb{Z}$ equipped with a binary operation $\ast$ given by 
\[
x \ast y = 2y - x
\]
becomes a quandle.
We call it the \emph{dihedral quandle} of order $q$, and write it as $R_{q}$.
It is routine to check that $R_{q}$ is connected if $q$ is odd.

Let us consider $R_{3}$-colorings $\mathscr{C}$ and $\mathscr{C}^{\prime}$ of a diagram of $8_{18}$, depicted in Figure \ref{R_3-colorings_of_8_18}.
Clearly, $\mathscr{C}$ is R-equivalent to $\mathscr{C}^{\prime}$.
Since the colors of the arc on the left and the arc at the top coincide under $\mathscr{C}$ but not under $\mathscr{C}^{\prime}$, there exists no inner automorphism $S$ of $R_{3}$ such that $\mathscr{C}^{\prime} = \mathscr{C} \cdot S$.
Thus, Theorem \ref{R-equivalence_and_inner_automorphism} does not yield a necessary condition for two colorings to be R-equivalent, in general.
On the other hand, we have the following theorem.
\begin{figure}[htbp]
 \centering
 \includegraphics[scale=0.5]{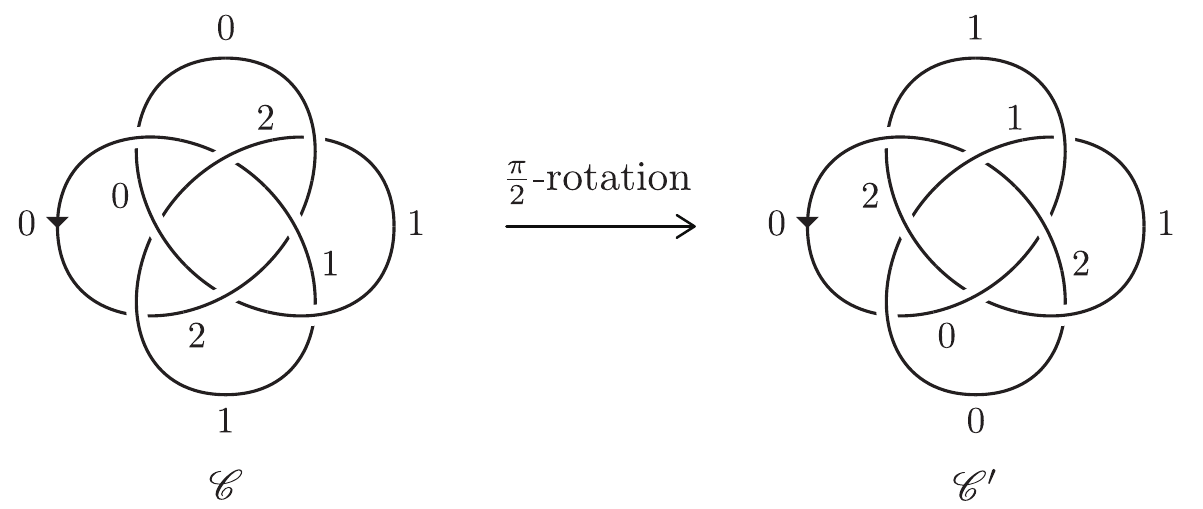}
 \caption{$R_{3}$-colorings of a diagram of $8_{18}$.}
 \label{R_3-colorings_of_8_18}
\end{figure}

Let $\alpha$ and $\beta$ be coprime odd integers with $\alpha$ positive and $- \alpha < \beta < \alpha$, and $S(\alpha, \beta)$ the $2$-bridge knot in Schubert's normal form (see \cite{K1996}).
Suppose that $D$ is the standard diagram of $S(\alpha, \beta)$, and $u_{1}$ and $u_{2}$ are the two upper-bridges of $D$.
Given an odd prime integer $p$, let us consider $R_{p}$-colorings of $D$.
As mentioned in \cite{I2005}, $D$ has non-trivial $R_{p}$-colorings if and only if there exists $k \in \mathbb{Z}$ such that $\alpha = k p$.
Furthermore, when $\alpha = k p$, each $R_{p}$-coloring of $D$ is uniquely determined by the colors $a$ and $b$ assigned to $u_{1}$ and $u_{2}$, respectively.
We denote such a coloring by $\mathscr{C}_{(a, b)}$.
We note that $\mathscr{C}_{(a, b)}$ is non-trivial if and only if $a \neq b$.

\begin{theorem}
 \label{necessary_and_sufficient_condition}
 Let $p$ be an odd prime integer and $\alpha = k p$ \textup{(}$k \in \mathbb{Z}$\textup{)}.
 Suppose that $\gcd(k, p) = 1$ if $p \geq 5$.
 Then non-trivial $R_{p}$-colorings $\mathscr{C}$ and $\mathscr{C}^{\prime}$ of $D$ are R-equivalent if and only if there exists $S \in \mathrm{Inn}(R_{p})$ such that $\mathscr{C}^{\prime} = \mathscr{C} \cdot S$.
\end{theorem}

By the argument given below, this theorem can be obtained as a consequence of the following.

\begin{theorem}
 \label{R-equivalence_classes_of_R_p-colorings_of_2-bridge_knot}
 Let $p$ be an odd prime integer and $\alpha = k p$ \textup{(}$k \in \mathbb{Z}$\textup{)}.
 Suppose that $\gcd(k, p) = 1$ if $p \geq 5$.
 Then any non-trivial $R_{p}$-coloring of $D$ is R-equivalent to exactly one of $\mathscr{C}_{(0, 1)}, \mathscr{C}_{(0, 2)}, \cdots, \mathscr{C}_{(0, \frac{p-1}{2})}$.
\end{theorem}

To prove Theorem \ref{R-equivalence_classes_of_R_p-colorings_of_2-bridge_knot}, we will use weights of shadow colorings defined using a quandle 3-cocycle.
For this purpose, we first review these notions briefly.
More details can be found in \cite{K2017}.

Let $X$ be a quandle, $D$ an oriented knot diagram, and $\mathscr{C}$ an $X$-coloring of $D$.
A map $\mathscr{R} : \{ \text{all regions of }D \} \to X$ is called a \emph{region coloring} of $D$ with respect to $\mathscr{C}$ if $\mathscr{R}$ satisfies the condition depicted in Figure \ref{region_coloring_condition} around each arc of $D$.
In the figure, $y$ denotes the color of the arc by $\mathscr{C}$, and $x$ and $x \ast y$ denote the elements of $X$ assigned to the corresponding regions by $\mathscr{R}$.
We call $x$ and $x \ast y$ the \emph{colors} of the regions.
The pair $(\mathscr{C}, \mathscr{R})$ is called a \emph{shadow coloring} of $D$ by $X$.
\begin{figure}[htbp]
 \centering
 \includegraphics[scale=0.5]{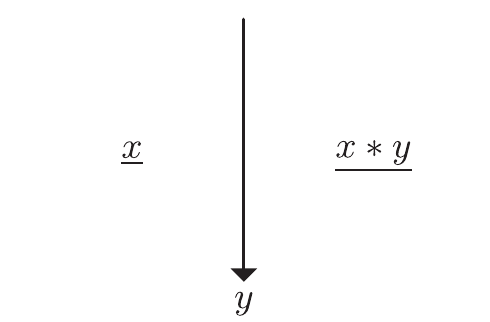}
 \caption{The condition for a region coloring.}
 \label{region_coloring_condition}
\end{figure}

Let $A$ be an abelian group.
A map $f : X^{3} \to A$ is said to be a \emph{quandle $3$-cocycle} of $X$ if $f$ satisfies the following two conditions.
\begin{itemize}
 \item[QC1.] For any $x, y, z, w \in X$,
 \begin{align*}
 & \hspace{1.4em} f(x, y, z) + f(x \ast z, y \ast z, w) + f(x, z, w) \\
 & = f(x \ast y, z, w) + f(x, y, w) + f(x \ast w, y \ast w, z \ast w).
 \end{align*}
 \item[QC2.] For any $x, y \in X$, $f(x, x, y) = f(x, y, y) = 0$.
\end{itemize}

Let $(\mathscr{C}, \mathscr{R})$ be a shadow coloring of $D$ by $X$, and $f : X^{3} \to A$ a quandle $3$-cocycle.
For each crossing $c$ of $D$, whose arcs and regions are colored by $(\mathscr{C}, \mathscr{R})$ as depicted in Figure \ref{local_weight}, we define the \emph{local weight} $W_{f}(c, \mathscr{C}, \mathscr{R})$ of $(\mathscr{C}, \mathscr{R})$ at $c$ by
\[
 W_{f}(c, \mathscr{C}, \mathscr{R}) = \varepsilon \cdot f(x, y, z),
\]
where $\varepsilon = + 1$ if $c$ is positive, otherwise $\varepsilon = - 1$.
Take the sum 
\[
 W_{f}(D, \mathscr{C}, \mathscr{R}) = \sum_{c} W_{f}(c, \mathscr{C}, \mathscr{R})
\]
of the local weights over all crossings of $D$, and call it the \emph{weight} of $(\mathscr{C}, \mathscr{R})$ associated with $f$.
\begin{figure}[htbp]
 \centering
 \includegraphics[scale=0.5]{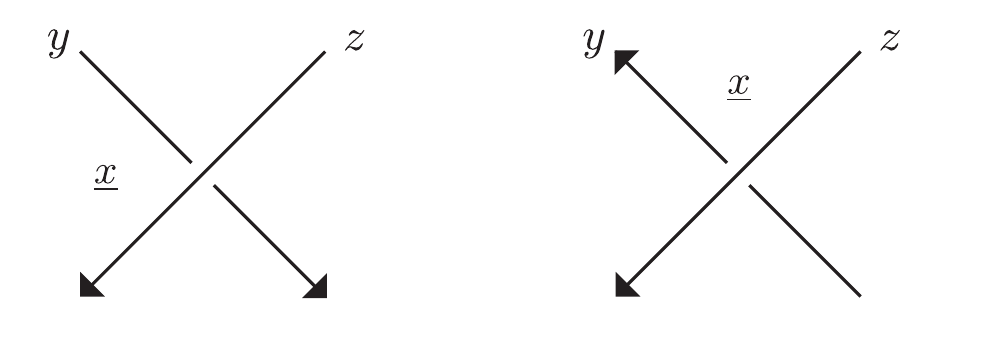}
 \caption{Colors of arcs and regions around a crossing $c$.}
 \label{local_weight}
\end{figure}

The following theorem plays a key role in the proof of Theorem \ref{R-equivalence_classes_of_R_p-colorings_of_2-bridge_knot}.

\begin{theorem}
 \label{invariant_of_R-equivalence}
 Let $X$ be a connected quandle and $f$ a quandle $3$-cocycle of $X$.
 Suppose that $X$-colorings $\mathscr{C}$ and $\mathscr{C}^{\prime}$ of an oriented knot diagram $D$ are R-equivalent to each other.
 Then, for any region colorings $\mathscr{R}$ and $\mathscr{R}^{\prime}$ of $D$ with respect to $\mathscr{C}$ and $\mathscr{C}^{\prime}$, respectively, $W_{f}(D, \mathscr{C}, \mathscr{R}) = W_{f}(D, \mathscr{C}^{\prime}, \mathscr{R}^{\prime})$.
 That is, $W_{f}(D, \mathscr{C}, \mathscr{R})$ is an invariant of R-equivalence.
\end{theorem}

\begin{proof}
 Take a finite sequence of Reidemeister moves and planar isotopies sending $\mathscr{C}$ to $\mathscr{C}^{\prime}$.
 Then, as mentioned in \cite{K2017}, the sequence sends $(\mathscr{C}, \mathscr{R})$ to a certain shadow coloring $(\mathscr{C}^{\prime}, \mathscr{R}^{\prime \prime})$, and we have $W_{f}(D, \mathscr{C}, \mathscr{R}) = W_{f}(D, \mathscr{C}^{\prime}, \mathscr{R}^{\prime \prime})$.
 Since $X$ is connected, $W_{f}(D, \mathscr{C}^{\prime}, \mathscr{R}^{\prime \prime}) = W_{f}(D, \mathscr{C}^{\prime}, \mathscr{R}^{\prime})$ by \cite[Theorem 3]{IK2014}.
 Thus, we have the claim.
\end{proof}

We next study $R_{p}$-colorings, where $p$ is an odd prime.

\begin{proposition}
 \label{R-equivalent_R_p-colorings}
 Let $p$ be an odd prime integer, $D$ an oriented knot diagram, and $\mathscr{C}$ a non-trivial $R_{p}$-coloring of $D$.
 For any $S \in \mathrm{Inn}(R_{p})$, $\mathscr{C}$ is R-equivalent to $\mathscr{C} \cdot S$.
\end{proposition}

\begin{proof}
 Since $\mathscr{C}$ is non-trivial, $\mathrm{Im} \, \mathscr{C}$ includes at least two distinct elements of $R_{p}$, and hence $\langle \mathrm{Im} \, \mathscr{C} \rangle = R_{p}$ by \cite{FNT2011}.
 Thus, we have the claim by Theorem \ref{R-equivalence_and_inner_automorphism}.\footnote{Under this assumption, $\mathrm{Im} \, \mathscr{C}$ generates the dihedral group of order $2p$, and hence the claim can also be obtained via \cite[Lemma 2.2]{M2011}.}
\end{proof}

Mochizuki \cite{M2003} gave the quandle $3$-cocycle $\theta_{p} : {R_{p}}^{3} \to \mathbb{Z} / p \mathbb{Z}$ defined by 
\[
\theta_{p}(x, y, z) = (x - y)\frac{(2z - y)^{p}+y^{p}-2z^{p}}{p}.
\]
We note that this formula was reformulated by Asami and Satoh \cite{AS2005}.
Let $\mathscr{R}$ be a region coloring of the standard diagram $D$ of $S(\alpha, \beta)$ with respect to $\mathscr{C}_{(a, b)}$.
Iwakiri \cite{I2005} gave the following formula:
\begin{align}
W_{\theta_{p}}(D, \mathscr{C}_{(a, b)}, \mathscr{R})
= \frac{\alpha}{p} n (a-b)(b-a)
= k n (a-b)(b-a), \label{eq:1}
\end{align}
where $n$ is the odd integer satisfying $1 \leq n \leq 2 \alpha - 1$ and $n \beta \equiv 1 \, (\hspace{-1.5ex} \mod 2 \alpha)$.
Since $R_{p}$ is connected, this value is an invariant of R-equivalence by Theorem \ref{invariant_of_R-equivalence}.
We write $W_{\theta_{p}}(D, \mathscr{C}_{(a, b)}, \mathscr{R})$ simply as $W_{(a, b)}$.

We are now ready to prove Theorem \ref{R-equivalence_classes_of_R_p-colorings_of_2-bridge_knot}.

\begin{proof}[Proof of Theorem \ref{R-equivalence_classes_of_R_p-colorings_of_2-bridge_knot}]
 For any integers $i$ and $j$ with $1 \leq i, j \leq p-1$, we have
 \[
 \mathscr{C}_{(i, i - j)} =
  \begin{cases}
  \mathscr{C}_{(0, j)} \cdot S_{\frac{i}{2}} & \text{if $i$ is even}, \\
  \mathscr{C}_{(0, j)} \cdot S_{\frac{i + p}{2}} & \text{if $i$ is odd}.
  \end{cases}
 \]
 Thus, $\mathscr{C}_{(0, j)}$ is R-equivalent to $\mathscr{C}_{(i, i - j)}$ by Proposition \ref{R-equivalent_R_p-colorings}.
 Similarly, for any integer $j$ with $1 \leq j \leq p-1$, $\mathscr{C}_{(0, j)}$ is R-equivalent to $\mathscr{C}_{(0, -j)}$ since $\mathscr{C}_{(0, -j)} = \mathscr{C}_{(0, j)} \cdot S_{0}$.
 Therefore, if $p = 3$, all non-trivial $R_{p}$-colorings of $D$ are R-equivalent to each other.
 If $p \geq 5$, any non-trivial $R_{p}$-coloring of $D$ is R-equivalent to at least one of $\mathscr{C}_{(0, 1)}, \mathscr{C}_{(0, 2)}, \cdots, \mathscr{C}_{(0, \frac{p-1}{2})}$.
 
 Suppose that $p \geq 5$.
 For any integers $j$ and $j^{\prime}$ with $1 \leq j, j^{\prime} \leq \frac{p-1}{2}$, 
 \begin{align}
  W_{(0, j^{\prime})} - W_{(0, j)} = k n (j + j^{\prime})(j - j^{\prime}) \label{eq:2}
 \end{align}
 by the formula (\ref{eq:1}).
 Since $\gcd(k, p) = \gcd(n, p) = 1$ and $1 \leq j, j^{\prime} \leq \frac{p-1}{2}$, $(\ref{eq:2}) \equiv 0 \, (\hspace{-1.5ex} \mod p)$ if and only if $j - j^{\prime} \equiv 0 \, (\hspace{-1.5ex} \mod p)$.
 It follows that $W_{(0, j^{\prime})} \neq W_{(0, j)}$ if $j^{\prime} \neq j$.
 Thus, $\mathscr{C}_{(0, j^{\prime})}$ is not R-equivalent to $\mathscr{C}_{(0, j)}$ if $j^{\prime} \neq j$.
\end{proof}

\section*{Acknowledgments}
\label{sec:acknowledgments}

The author would like to express her sincere gratitude to her supervisor Professor Ayumu Inoue for his hearty encouragements and helpful suggestions.

\bibliographystyle{amsplain}

\end{document}